\documentclass[extended]{article}
\usepackage{amssymb,amsmath,amsfonts,amsthm,enumerate,color}
\usepackage[toc,page,title,titletoc,header]{appendix}
\usepackage{graphicx}
\usepackage{indentfirst}
\usepackage{multicol}
\usepackage{booktabs}
\usepackage{url}
\usepackage{amscd}
\usepackage{wrapfig}
\usepackage{upgreek}

\numberwithin{equation}{section}

\newtheorem{theorem}{Theorem}

\newtheorem{assumption}{Assumption}
\newtheorem{definition}{Definition}
\newtheorem{corollary}{Corollary}
\newtheorem{remark}{Remark}
\newtheorem{lemma}{Lemma}

\def \Vh0{\stackrel{\circ}{V}_h} \def\to{\rightarrow}

 \def\v{{\bf v}}

\newcommand{\lc}
{\mathrel{\raise2pt\hbox{${\mathop<\limits_{\raise1pt\hbox
{\mbox{$\sim$}}}}$}}}

\newcommand{\gc}
{\mathrel{\raise2pt\hbox{${\mathop>\limits_{\raise1pt\hbox{\mbox{$\sim$}}}}$}}}

\newcommand{\ec}
{\mathrel{\raise2pt\hbox{${\mathop=\limits_{\raise1pt\hbox{\mbox{$\sim$}}}}$}}}

\def\bb{\begin{equation}} \def\ee{\end{equation}}

\def\beqn{\begin{eqnarray}}  \def\eqn{\end{eqnarray}}

\def\beqnx{\begin{eqnarray*}} \def\eqnx{\end{eqnarray*}}

\def\bn{\begin{enumerate}} \def\en{\end{enumerate}}

\def\bd{\begin{description}} \def\ed{\end{description}}
\DeclareMathOperator*{\argmin}{arg\,min}


\newcommand{\EE}{\mathbb{E}}
\newcommand{\RR}{\mathbb{R}}
\newcommand{\prox}{\textbf{Prox}}
\newcommand{\Min}{\textrm{minimize}}
\newcommand{\PP}{\mathbb{P}}

\begin{document}

\title{Markov Chain Block Coordinate Descent}
\author{Tao Sun\thanks{
Department of Mathematics, National University of Defense Technology,
Changsha, 410073, Hunan,  China. Email: \texttt{nudtsuntao@163.com}. }
\and Yuejiao Sun\thanks{
  Department of Mathematics, UCLA, 601 Westwood Plz, Los Angeles, CA 90095, USA, Email: \texttt{sunyj@math.ucla.edu}.}
  \and Yangyang Xu\thanks{
 Department of Mathematical Sciences,
Rensselaer Polytechnic Institute, Troy, NY 12180, USA, Email: \texttt{xuy21@rpi.edu}.}
  \and   Wotao Yin\thanks{
  Department of Mathematics, UCLA, 601 Westwood Plz, Los Angeles, CA 90095, USA, Email: \texttt{wotaoyin@math.ucla.edu}, corresponding author.}
}

\maketitle

\begin{abstract}
The method of block coordinate gradient descent (BCD) has been a powerful method for large-scale optimization. This paper considers 
the BCD method that successively updates a series of blocks selected according to a Markov chain.
This kind of block selection is neither i.i.d. random nor cyclic. On the other hand, it is a natural choice for some applications in distributed optimization and Markov decision process, where i.i.d. random and cyclic selections are either infeasible or very expensive.
By applying mixing-time properties of a Markov chain, we prove convergence of Markov chain BCD for minimizing Lipschitz differentiable functions, which can be nonconvex. When the functions are convex and strongly convex, we establish both sublinear and linear convergence rates, respectively. We also present a method of Markov chain inertial BCD. Finally, we discuss potential applications.
\end{abstract}


\textbf{AMS}  Primary: 90C26; secondary:  	90C40, 68W15
\tableofcontents

\section{Introduction}
We consider
 the following  minimization problem
 \begin{equation}\label{model}
    \Min~~f(x)\equiv f(x_1,x_2,\ldots,x_N)
 \end{equation}
 where $f:\mathbb{R}^N\mapsto \mathbb{R}$ is a differentiable  function (possibly nonconvex) and every $\nabla_i f$  ($i=1,2,\dots,N$) is Lipschitz with constant $L>0$.

The block coordinate gradient descent (BCD) method is a popular approach that can take the advantage of the coordinate structure in \eqref{model}. The method updates one coordinate, or a block of coordinates, at each iteration, as follows. For $k=0,1,\ldots$, choose $i_k\in [N]:=\{1,2,\ldots,N\}$ and compute
 \begin{equation}\label{y-CD}
 x^{k+1}_{i_k} =x^k_{i_k} - \gamma \nabla_{i_k} f(x^{k}),
\end{equation}
where $\gamma$ is a step size; for remaining $j\in[N]\backslash\{ i_k\}$, we keep $x_j^{k+1}=x_j^k$.

The coordinate gradient descent method was introduced in \cite{TsengYun2009_coordinate}. The random selection rule (i.i.d. over the iterations) appeared in \cite{shalev2011stochastic,nesterov2012efficiency}. In the same paper \cite{nesterov2012efficiency}, the method of accelerated coordinate gradient descent was proposed, and it was later analyzed in \cite{LeeSidford2013_efficient} for both convex and strongly convex functions. Both \cite{nesterov2012efficiency,LeeSidford2013_efficient} select a coordinate $i$ with probability proportional to the Lipschitz constant $L_i$ of $g(\alpha)=\nabla_i f(x+\alpha e_i)$ over free $x$; the rate is optimal when $L_i$'s are equal. An improved random sampling method with acceleration was introduced in~\cite{Allen-ZhuQuRichtarikYuan2015_EvenFaster}, which further decreases the complexity when some $L_i$'s are significantly smaller than the rest. This method was further generalized in~\cite{HannahFengYin2018_a2bcd} to an asynchronous parallel method, which obtains parallel speedup to the accelerated rate.  
In another line of work, \cite{DangLan2015_stochastic} combines stochastic coordinate gradient descent with mirror descent stochastic approximation, where a random data mini-batch is taken to update a randomly chosen coordinate. This is improved in \cite{XuYin2015_block}, where the presented method uses each random mini-batch to update all the coordinates in a sequential fashion.
Besides stochastic selection rules, there has been work of the cyclic sampling rule. The work~\cite{XuYin2013_block} studies its convergence under the convex and nonconvex settings, and \cite{BeckTetruashvili2013_convergence} proves sublinear and linear rates in the convex setting. The constants in these rates are worse than standard gradient descent though. For a family of problems, \cite{sun2015improved} obtains improved rates to match standard gradient descent (and their results also apply to the random shuffling rule).
The greedy sampling rule has also been studied in the literature but unrelated to this paper. Let us just mention some references~\cite{LiOsher2009_coordinate,PengYanYin2013_parallel,LiUschmajewZhang2015_convergence,NutiniSchmidtLaradjiFriedlanderKoepke2015_coordinate}.
Finally, \cite{PengWuXuYanYin2016_coordinate}~explores the family of problems with the structure that enables us to update a block coordinate at a much lower cost than updating all blocks in batch.


This paper introduces the Markov-chain select rule. We call our method \emph{Markov-chain block gradient coordinate descent} (MC-BCD). In this method, $i_k$ is selected according to a Markov chain; hence, unlike the above methods, our choice is \emph{neither} stochastic i.i.d. (with respect to $k$) \emph{nor} deterministic. Specifically, there is an underlying strongly-connected graph
$\mathcal{G} = (\mathcal{V}, \mathcal{E})$ with the set of vertices $\mathcal{V}:=[N]$ and set of edges $\mathcal{E}\subseteq \mathcal{V}\times \mathcal{V}$.
Each node $i\in \mathcal{V}$ can compute $\nabla_i f(\cdot)$ and update $x_i$. We call $(i_k)_{k\ge 0}$ a walk of $\mathcal{G}$ if every $(i_k,i_{k+1})\in\mathcal{E}$. If the walk $(i_k)_{k\ge 0}$ is deterministic and visits every node at least once in every $K$ iterations, then $(i_k)_{k\ge 0}$ is essentially cyclic; if every $i_{k+1}$ is chosen randomly from  $\{\text{neighbors of }i_k\}\cup\{i_k\}$, then we obtain MC-BCD, which is the focus of this paper. To the best of our knowledge, MC-BCD is new.


\subsection{Motivations}\label{reduc}
Generally speaking, one does not use MC-BCD to accelerate i.i.d. random or cyclic BCD but for other reasons: 
When we are forced to take Markov chain samples because cyclic and stochastic samples are not available; Or, although   cyclic and stochastic samples are available, it is easier or cheaper to take Markov chain samples. We briefly present some examples below to illustrate those motivations. Some examples are tested numerically in Section 6 below.



\textbf{Markov Chain Dual Coordinate Ascent (MC-DCA).}
The paper~\cite{shalev2013stochastic} proposes the Stochastic Dual Coordinate Gradient Ascent (SDCA) to solve
\begin{align}\label{du-prim}
    \Min_{w\in\mathbb{R}^d}\{\frac{\lambda}{2}\|w\|^2+\frac{1}{N}\sum_{i=1}^N \ell_{i}(w^{\top} a_i)\},
\end{align}
where $\lambda>0$ is the regularization parameter, $a_i$ is the data vector associated with $i$th sample, and $\ell_i$ is a convex loss function. Its dual problem can be formulated as
\begin{align}\label{dual}
    \Min_{\alpha\in \mathbb{R}^{N}}\{D(\alpha):=\frac{\lambda}{2}\|A\alpha\|^2+\frac{1}{N}\sum_{i=1}^N \ell_{i}^*(-\alpha_i)\},
\end{align}
where $A\in\mathbb{R}^{d\times N}$ with column $A_i:=\frac{a_i}{\lambda N}$, and $\ell_i^*$ is the conjugate of $\ell_i$. By applying stochastic BCD to \eqref{dual}, SDCA can reach comparable or better convergence rate than stochastic gradient descent. We employ this idea and propose MC-DCA : in the $k$th iteration,  while $x_j^{k+1}=x_j^k$ if $j\in[N]\backslash\{ i_k\}$,
\begin{align}\label{mcdca}
    \alpha^{k+1}_{i_k}=\alpha^k_{i_k}-\gamma\Big(\lambda A_{i_k}^{\top} (A \alpha^k)-\frac{\nabla \ell_{i_k}^*(-\alpha_{i_k}^k)}{N}\Big)
\end{align}
where  $(i_k)_{k\ge 0}$ is a Markov chain.

The Markov chain must come from somewhere. Consider that the data $a_1, a_2, \ldots, a_N$ are stored in a distributed fashion over a graph. Only when the graph is complete can we efficiently sample $i_k$ i.i.d. randomly and access $a_{i_k}$; only when the graph has a Hamiltonian cycle can we visit the data in a cyclic fashion without visiting any node twice in each cycle. MC-DCA works under a much weaker assumption: as long as the graph is connected.
Specifically, let a token hold $(\alpha_1,\alpha_2,\ldots,\alpha_N)$ and vector $A\alpha$, and let the token randomly walk through the nodes in the network; each node $i$ holds data $A_i$ and can compute $\nabla \ell_i^*$;
as the token arrives at node $i$, the node accesses $(A\alpha)$ and $\alpha$ and computes  $\lambda A_{i}^{\top} (A \alpha)$ and  $\nabla \ell_i^*(-\alpha_{i})$, which are used to update $\alpha_i$ and update $(A\alpha)$.

\textbf{Future rewards in a Markov decision process.}
This example is a finite-state ($N$ states) Discounted Markov Decision Process (DMDP) for which we can compute the transition probability from any current state $i$ to the next state, or quickly approximate it. We can use MC-BCD to compute the expected future reward vector.

Let us describe the DMDP. Entering any state $i$, we receive an award $r_i$ and then take an action according to a given policy $\pi$ (a state-to-action mapping). After the action is taken, the system enters a state $j$, $j\in[N]$, with probability $P_{i,j}$. The transition matrix $P:=[P_{i,j}]_{i,j\in[N]}$ depends on the action taken and thus depends on $\pi$. The reward discount factor is $\gamma \in(0,1)$.
Our goal is to evaluate the expected future rewards of all states $i\in [N]$ for fixed $\pi$. This step dominates the per-step computation of the policy-update iteration \cite{sutton1998reinforcement}, which iteratively updates $\pi$. 

For each state $i_0 := i$, the expected future reward is given as
$$v_i:=\EE_{\{i_t\}}\Big[\sum_{t=1}^{+\infty}\gamma^t r_{i_t}\mid i_0 = i \Big],$$ where the state sequence $(i_t)_{t\ge 0}$ is a Markov chain induced by the transition matrix $P$ and $r_{i_t}$ is the reward received at time $t$. The corresponding Bellman equation is $v_i = \EE_{i_1}\big[r_i + \gamma v_{i_1}\mid i_0 = i\big] = r_i + \gamma\sum_{j\in [N]}P_{i,j}v_j$, the matrix form of which is
\begin{align}\label{bellman}
    v=r+\gamma P v,
\end{align}
where $v=[v_1, v_2,\ldots, v_N]^T$ and $r=[r_1, r_2,\ldots, r_N]^T$. 

When $N$ is huge, solving \eqref{bellman} is difficult. Often we have memory to store a few $N$-vectors (also, $N$ can be reduced by dimension reduction) but not an $N\times N$-matrix. Therefore, we can store the vector $P_i = [P_{i,1},\dots,P_{i,N}]^T$ only temporarily in each iteration. In the case where the physical principles or the rule of game are given, such as in the Tetris game, we can compute the transition probabilities $P_{i_k}$ explicitly. 
Consider another scenario where $P_{i_k}$ can not be computed explicitly but can be approximated by Monte-Carlo simulations. The simulation of transition at just one state $i_k$ is much cheaper than that of all states. In both scenarios, we have access to $P_{i_k}$. This allows us to apply MC-BCD to solve a dual optimization problem below to compute the future reward vector $v$,
\begin{align}\label{bellman2}
    \Min_{v}\, \{\frac{1}{2N}\|(\mathbb{I}_{N}-\gamma P)v-r\|^2+\frac{\lambda}{2}\|v\|^2\},
\end{align}
where $\lambda\geq 0$ is a fixed regularization parameter.
This corresponds to setting $A:=\mathbb{I}_{N}-\gamma P$ in \eqref{dual}. Note that in DMDP, one cannot transit from the current state $i_k$ to an arbitrary $j\in[N]$. Therefore, standard cyclic and  stochastic BCD is not applicable.

Running the MC-DCA iteration \eqref{mcdca} requires the vectors $A_{i_k} = P_{i_k}$ and $A\alpha^k = \alpha^k - \gamma P\alpha^k$. We update $(P\alpha^k)$ by maintaining a sequence $(w^k)_{k\ge 0}$ as follows: initialize $\alpha^0:=0$ (vector zero) and thus $w^0 = P\alpha^0 = 0$; in $k$th iteration, we compute $w^{k+1} := w^{k} + P_{i_k}(\alpha^{k+1}-\alpha^{k})_{i_k} = P\alpha^{k+1}$, where the equality follows since $\alpha^{k+1}$ and $\alpha^{k}$ only differ over their $i_{k}$th component. This update is done without accessing the full matrix $P$.

As we showed above, running our algorithm to compute the expected future award $v$ only requires $O(N)$ memory. Also the algorithm iterates simultaneously while the system samples its state trajectory. Suppose each policy $\pi$ can be stored in $O(N)$ memory (e.g., deterministic policy) and updating $\pi$ using a computed $v$ also needs $O(N)$ memory; then, we can a policy-update iteration with $O(N)$ memory.

\textbf{Risk minimization by dual coordinate ascent over a tricky distribution.}
Let $\Xi$ be a statistical sample space with distribution $\Pi$, and  $F(\cdot):\mathbb{R}\rightarrow \mathbb{R}$ is a proper, closed, strongly convex function.
Consider the following regularized expectation minimization problem
\begin{align}\label{regerm}
    \Min_{w\in \mathbb{R}^n} ~~\EE_{\xi}\big(F(w^{\top}\xi)\big)+\frac{\lambda}{2}\|w\|^2
\end{align}
Since the objective is strongly convex, its dual problem is smooth. If it is easy to sample data from $\Pi$, \eqref{regerm} can be solved by SDCA, which uses i.i.d. samples.
When the distribution $\Pi$ is difficult to sample directly but has a faster Markov Chain Monte Carlo (MCMC) sampler, 
we can apply MC-DCA to this problem.  

\textbf{Multi-agent resource-constrained optimization.} Consider the  multi-agent optimization problem of $n$ agents \cite{brucker1999resource}:
\begin{align}\label{multi-agentopt}
    \Min\, f(x_1,x_2,\ldots,x_N)+\frac{\beta}{2}\|\max\{Ax - b, \textbf{0}\}\|^2,
\end{align}
where $f$ is the cost function, $b$ is the resource vector, and  $\max\{Ax - b, \textbf{0}\}$ penalizes any over usage of resources. Define a graph, in which every node is an agent and every edge connects a pair of agents that either depend on one another in $f$ or share at least one resource. In other words, the objective function \eqref{multi-agentopt} has a graph structure in that computing the gradient of $x_i$  requires only the information of the adjacent agents of $i$.

MC-BCD becomes a {decentralized} algorithm: after an agent $i_k$ updates its decision variable $x_{i_k}$, it broadcasts $x_{i_k}$ to one of its neighbors, $i_{k+1}$ and activates it to run next step. 
In this process, $i_0,i_1,\dots$ form a random walk over the graph and, therefore, is a Markov chain. As long as the network is connected, a central coordinator is no more necessary. However, sampling $i_k$ i.i.d. randomly requires a central coordinator and will consume more communication since it may communicate beyond neibors. Also selecting $i_k$ essentially cyclicly  requires a tour of the graph, which relys on the knowledge of the graph topology.

 When $f$ is differentiable with Lipschitz continuous gradient, so is the objective function. We apply MC-BCD to \eqref{multi-agentopt} to obtain
\begin{align}\label{schem-mcbc-multi-agent}
    x^{k+1}_{i_k}=x^k_{i_k}-\gamma \nabla f_{i_k}(x^k)-\gamma \beta A_{i_k}^{\top}\max\{Ax^k-b,\textbf{0}\},
\end{align}
where $(i_k)_{k\geq}\subseteq [N]$ is a Markov chain.
We assume that agent $i$ can access $A_i$ and $b_i$ and compute $\nabla_i f$. Similar to the example for computing expected future reward above, $v^k:=Ax^k-b$ can be updated along with the iterations so no node needs the access to the full matrix $A$. Alternatively, we can use a central governor which receives updated $x^k$ and $v^k$ from agent $i_k$ and sends the data to $i_{k+1}$ for the next iteration.  

\textbf{Decentralized optimization.}
This example is taken from \cite{yin2018communication}. Again consider the empirical risk minimization problem \eqref{du-prim}.
We consider solving its dual problem \eqref{dual} in a network by assigning each sample $a_i$ 
to a node.
A parallel distributed algorithm will update for all the components, $i=1,...,N$, concurrently.

If the network has a central server, then each node sends its  $\alpha_i$ to the central server, which forms $A\alpha=\sum_{i=1}^n A_i \alpha_i$ and then broadcasts it back to the nodes.

If the network does not have a central server, then we can form $A\alpha$ either running a decentralized gossip algorithm or calling an all-reduce communication. The former does not require the knowledge of the network topology and is an iterative method. The latter requires the topology and takes at least $O(\log N)$ rounds and at least $O(N)$ total communication, even slower when the network is sparse. An alternative approach is to create a token that holds $A\alpha$ and follows a random walk in the network. The token acts like a traveling center. When the token arrives at a node $i_k$, the node updates its $\alpha_{i_k}$ using the token's $A\alpha$, and this local update leads to a sparse change to $A\alpha$; updating $A\alpha$ requires no access to $\alpha_j$ for $j\neq i_k$. The method in \cite{yin2018communication} applies this idea to an ADMM formulation of the decentralized consensus problem (rather than BCD in this paper) and shows that total communication is significantly reduced.



\subsection{Difficulty of the convergence proofs: biased expectation}
Sampling according to a Markov chain is neither (essential) cyclic nor i.i.d. stochastic. No matter how large $K$ is, it is still possible that a node is never visited during some $k+1,...,k+K$ iterations. Unless the graph $\mathcal{G}$ is a complete graph (every node is directly connected with every other node), there are nodes $i,j$ \emph{without} an edge connecting them, i.e., $(i,j)\not\in \mathcal{E}$. Hence, given $i_{k-1} = i$, it is \emph{impossible} to have $i_{k}=j$. So, no matter how one selects the sampling probability $p_k=\PP(i_k = p_k)$ and step size $\gamma_k$, we generally do \emph{not} have
$     \EE_{i_k}(\gamma_k\vec{\nabla}_{i_k} f(x^k)\mid i_{k-1}=i)= C \nabla f(x^k)$
for any constant $C$, where $\vec{\nabla}_{i_k} f(x^k):=[0,\dots,0,{\nabla}_{i_k} f(x^k),0,\dots,0]^T$. This, unfortunately, breaks down all the existing analyses of stochastic BCD since they all need a non-vanishing probability for every block $1,\ldots,N$ to be selected.



\subsection{Proposed method and contributions}
Given a graph $\mathcal{G}=(\mathcal{V},\mathcal{E})$, MC-BCD is written mathematically as
\begin{subequations}\label{MC-BCD}
\begin{align}
\label{SP}
 \text{sample}~& i_{k}\in\{j:(i_{k-1},j)\in \mathcal{E}\} \sim P_{i_{k-1},j}(k),\\
\label{CD}
\text{compute}~& x^{k+1}_{i_k} = x^k_{i_k} - \gamma \nabla_{i_k} f(x^{k}),
\end{align}
\end{subequations}
where $\gamma$ is a constant stepsize, and $P(k)$ is the transition matrix in the $k$th step (details given in Sec. 2), 
and we maintain $x^{k+1}_j = x^k_j$ for all $j\not=i_k$.
The initial point $x^0$ can be chosen arbitrarily. The block $i_0$ can be chosen either deterministically or randomly.
The following diagram illustrates the influential relations of $x^0$ and the random variable sequences $(i_k)_{k\ge 0}$ and $(x^k)_{k\ge 1}$:
\begin{equation}\label{tutu}
    \CD
  @.       i_0 @>  >> i_1 @> >> i_2 @> >> i_3 @> >> \ldots\\
  @.       @V  VV @V  VV @V  VV @V  VV @.   \\
  x^0 @> >> x^1 @> >> x^2 @> >> x^3 @> >> x^4 @>>>\ldots
    \endCD
\end{equation}

To our best knowledge, \eqref{MC-BCD} did not appear before and, as explained above, is not a special case of existing BCD analyses. When the Markov chain $(i_{k})_{k\ge 0}$ has a finite mixing time and problem \eqref{model} has a lower bounded objective, we show that using $\gamma\in(0,2/L)$ ensures $\EE\|\nabla f(x^k)\|\to 0$. The concept of mixing time is reviewed in the next section. In addition, when $f$ is convex and coercive, we show that $\EE f(x^k) \to \min f$ at the rate of $O(1/k)$ with a hidden constant related to the mixing time. 
Note that running the algorithm itself requires no knowledge about the mixing time of the chain. Furthermore, when $f$ is (restricted) strongly convex, then the rate is improved to be linear, unsurprisingly. Although we do not develop any Nesterov-kind acceleration in this paper, a heavy-ball-kind inertial MC-BCD is presented and analyzed because the additional work is quite small. When the computation $\nabla_{i_k} f(x^{k})$ is noisy, as long as the noise is square summable (which is weaker than being summable), MC-BCD still converges.

\subsection{Possible future work}
We mention some future improvements of MC-BCD, 
which will require significantly more work to achieve. First, 
it is possible to accelerate MC-BCD 
using both Nesterov-kind momentum and optimizing the transition probability. 
Second, it is important to parallelize MC-BCD, for example, to allow multiple random walks to simultaneously update different blocks~\cite{RichtarikTakac2016_parallel,FercoqRichtarik2015_accelerated}, even in an asynchronous fashion like \cite{liu2015asynchronous,PengXuYanYin2016_convergence,sun2017asynchronous}. 
Third, it is interesting to develop a primal-dual type MC-BCD, which would apply to a model-free DMDP along a single trajectory. Yet another line of work applies block coordinate update to linear and nonlinear fixed-point problems \cite{peng2016arock,PengWuXuYanYin2016_coordinate,chow2017cyclic} because it can solve optimization problems in imaging and conic programming, which are equipped with nonsmooth, nonseparable objectives, and constraints.

\section{Preliminaries}
\subsection{Markov chain}
We recall some definitions and propertiesof the 
Markov chain that we use in this paper. 
\begin{definition}[finite-state (time-homogeneous) Markov chain]
A stochastic process $X_1,X_2,...$ in a finite state space $[N]:=\{1,2,\ldots,N\}$ is called  Markov chain with transition matrices $(P(k))_{k\geq 0}$ if, for $k\in \mathbb{N}$, $i,j\in [N]$, and  $i_0,i_1,\ldots,i_{k-1}\in [N]$, we have
\begin{equation}
    \mathbb{P}(X_{k+1}=j\mid X_0=i_0,X_1=i_1,\dots,X_k=i)=\mathbb{P}(X_{k+1}=j\mid X_k=i)=P_{i,j}(k).
\end{equation}
The chain is time-homogeneous if $P(k)\equiv P$ for some constant matrix $P$.
\end{definition}
Let the probability distribution of $X_k$ be denoted as the row vector $\pi^k=(\pi^k_1,\pi^k_2,\ldots,\pi^k_N)$, that is, $\PP(X_k=j)=\pi_j^k$. Each $\pi^k$ satisfies  $\sum_{i=1}^N \pi^k_i=1.$ Obviously, it holds $\pi^{k+1}=\pi^k P(k)$.
When the Markov chain is time-homogeneous, we have $\pi^k=\pi^{k-1} P$ and
$\pi^k=\pi^{k-1} P=\cdots=\pi^0 P^{k}$,
for $k\in \mathbb{N}$, where $P^{k}$ is the $k$th power of $P$.
\begin{definition}
A time-homogeneous Markov chain is irreducible if, for any $i,j\in [N]$, there exists $k$ such that $(P^k)_{i,j}>0$.
State $i\in[N]$ is said to have a period $d$ if $P^k_{i,i} = 0$ whenever $k$ is \emph{not} a
multiple of $d$ and $d$ is the greatest such integer. If $d=1$, then we say state $i$ is aperiodic. If  every state is aperiodic, the Markov chain is said to be aperiodic.
\end{definition}

Any time-homogeneous, irreducible, and aperiodic
Markov chain has a stationary distribution $\pi^*=\lim_k \pi^k
=[\pi^*_1,\pi^*_2,\ldots,\pi^*_N]$ with $\sum_{i=1}^N \pi^*_i=1$ and $\min_i\{\pi^*_i\}>0$, and $\pi^*= \pi^* P$. This is a sufficient but not necessary condition to have such $\pi^*$. If the Markov fails to be time-homogeneous\footnote{The time-homogeneous, irreducible, and aperiodic
Markov chain is widely used; however, in practical problems, the Markov chain may not satisfy the time-homogeneous assumption.
For example, in a mobile, if the network connectivity structure is changing all the time, then the set of the neighbors of an agent is time-varying \cite{johansson2007simple}.}, it may still have a stationary distribution under additional assumptions.


In this paper, we make the following assumption, which always holds for time-homogeneous, irreducible, and aperiodic Markov chain and may hold for more general Markov chains.

\begin{assumption}\label{ass:mc}{The Markov chain $(X_k)_{k \ge 0}$ has the transition matrices $(P(k))_{k \ge 0}$ and the stationary distribution $\pi^*$. Define
\begin{align*}
    \Phi(m,n):=P(m)P(m+1)\cdots P(m+n),\quad m,n\geq 0,\qquad
\Pi^*:=\begin{bmatrix}
                    \pi^* \\
                    \pi^* \\
                    \vdots \\
                    \pi^*
              \end{bmatrix}\in\mathbb{R}^{N\times N},
\end{align*}
that is, every row of $\Pi^*$ is $\pi^*$. For each $\epsilon >0$, there exists $\uptau_\epsilon\ge 1$ such that
\begin{equation}\label{convermatrix-g}
    \|\Phi(m,n)-\Pi^*\|_2< \epsilon,\quad \text{for all } m\ge 0, n\ge \uptau_\epsilon-1.
\end{equation}}
\end{assumption}
Here, $\uptau$ is called a \emph{mixing time}, which specifies how long a Markov chain evolves close to its stationary distribution. The literature has a thorough investigation of various kinds of mixing times \cite{bradley2005basic}.
Previous mixing time focuses on bounding the difference between $\pi^k$ and the stationary distribution $\pi^*$. Our version is just easier to use in the analysis.


For a time-homogeneous, irreducible, and aperiodic Markov chain with the transition matrix $P$, $\Phi(m,n)=P^{n+1}$. It is easy to have $\uptau_\epsilon$ as $(1+\frac{3\ln N}{2\ln\frac{1}{\lambda_2(P)}})\cdot\log_{\frac{1}{\lambda_2(P)}}(\frac{1}{\epsilon})$, where $\lambda_2(P)$ denotes the second largest eigenvalue of $P$ (positive and smaller than 1) \cite{meyn2012markov}.
Besides the time-homogeneous, irreducible, and aperiodic Markov chain, some other non-time-homogeneous chains can also have  a geometrically-convergent $\Phi(m,n)$. An example is presented in \cite{ram2009incremental}.

\subsection{Notation and constants}
The following notation is used throughout this paper:
\begin{align}\label{eq:def-delta}
\Delta^k:=x^{k+1}-x^k
\end{align}
In MC-BCD iteration, only the block $\Delta^k_{i_k}$ of $\Delta^k$ is nonzero; other blocks are zero.
Let $\pi^*_{\min}$ be the minimal stationary distribution, i.e.,
\begin{equation}\label{pim}
    \pi^*_{\min}:=\min_{1\leq i\leq N}\{\pi_i^{*}\}.
\end{equation}
For any closed proper function $f$, $\argmin f$ denotes  the set $\{x\in \mathbb{R}^{N}\mid f(x)=\min f\}$, and $\|\cdot\|$ denotes the $\ell_2$ norm. Through the proofs, we use the following sigma algebra
$$\chi^k:=\sigma(x^1,x^2,\ldots,x^k,i_0,i_1,\ldots,i_{k-1}).$$

Let Assumption \ref{ass:mc} hold. In our proofs, we let $\uptau$ be the $\frac{\pi^*_{\min}}{2}$-mixing time, i.e.,
\begin{align}\label{mixt1}
    \|\Phi(m,n)-\Pi^*\|_2\leq\frac{\pi^*_{\min}}{2},~~\text{whenever}~~n\geq \uptau-1.
\end{align}
With  direct calculations,
\begin{equation}\label{th1-t0}
    \frac{\pi^*_{\min}}{2}\leq [\Phi(m,n)]_{i,j},\textrm{~for~any~}i,j\in\{1,2,\ldots,N\}, n\geq \uptau-1.
\end{equation}
If the Markov chain promises  a geometric rate, then we have
 \begin{equation}\label{y-jianduan}
    \uptau=O\left(\ln\frac{2}{\pi^*_{\min}}\right).
\end{equation}
It is worth mentioning that, for a complete graph where all nodes are connected to each other, we have a Markov chain with $\uptau=1$, and our MC-BCD will reduce to random BCD~\cite{nesterov2012efficiency}.

\section{Markov chain block coordinate gradient descent}\label{sec:mc-bcd}
In this section, we study the convergence properties of the MC-BCD for problem (\ref{model}). The discussion covers both convex and nonconvex cases. We show that the MC-BCD can converge if the stepsize $\gamma$ is taken as the same as that in traditional BCD. For convex problems, sublinear convergence rate is established, and for strongly convex cases, linear convergence is shown.

Our analysis is conducted to an inexact version of the MC-BCD, which allows error in computing partial gradients: 
\begin{equation}\label{inCD}
 x^{k+1}_{j} = \left\{\begin{array}{ll}
 x^k_{j} - \gamma\big( \nabla_{j} f(x^{k})+\epsilon^k\big),&\text{ if }j=i_k\\[0.2cm]
 x^k_j, & \text{ if }j\neq i_k,
 \end{array}
 \right.
\end{equation}
where $i_k$ is sampled in the same way as in \eqref{SP}, and $\epsilon^k$ denotes the error in the $k$th iteration. If $\epsilon^k$ vanishes, the above updates reduce to the MC-BCD in \eqref{MC-BCD}.
\subsection{Convergence analysis}
%
The results in this section applies to both convex and nonconvex cases, and they rely on the following assumption.
\begin{assumption} \label{ass:2}
The set of minimizers of function $f$ is nonempty, and $\nabla_i f$ is Lipschitz continuous about $x_i$ with constant $L> 0$ for each $i=1,2,\dots,N$, namely,
\begin{equation}\label{eq:lip-i}\|\nabla_i f(x) - \nabla_i f(x+\alpha e_i)\|\le L\|\alpha\|,\quad \forall x\in\RR^N, \forall \alpha\in\RR,
\end{equation}
where $e_i$ denotes the $i$th standard basis vector in $\RR^N$. In addition, $\nabla f$ is also Lipschitz continuous about $x$ with constant $L_r$, namely,
\begin{equation}\label{eq:lip-f}\|\nabla f(x) - \nabla f(x+s)\|\le L_r\|s\|,\quad  \forall x\in\RR^N, \forall s\in\RR^N.
\end{equation}
We call $\kappa=\frac{L_r}{L}$  the condition number.
\end{assumption}
When \eqref{eq:lip-i} holds for each $i$, we have
\begin{equation}\label{eq:lip-ineq}
f(x+de_i)\le f(x)+\langle \nabla_i f(x), de_i\rangle+\frac{L}{2}\|d\|^2.
\end{equation}

Lemma \ref{lem:bd-Delta} below is very standard. It bounds the square summation of $\Delta^k$ by initial objective error and iteration errors. Lemmas \ref{lem:bd-gik} and \ref{lem:cbd-gik} are new; they study the bounds on $\|\nabla_{i_k}f(x^{k-\uptau + 1})\|^2$ because the sampling bias prevents us from directly bounding $\|\nabla_{i_k}f(x^{k})\|^2$. The bounds in these three lemmas are combined in Theorem \ref{th1} to get the convergence rates of $\|\nabla f(x^k)\|$.
\begin{lemma}\label{lem:bd-Delta}
Under Assumption \ref{ass:2}, let $(x^k)_{k\geq 0}$ be generated by the inexact MC-BCD \eqref{inCD} with any constant stepsize $0<\gamma<\frac{2}{L}$. Then for any $k$,
\begin{equation}\label{eq:sum-dx}
\sum_{t=0}^k\|\Delta^t\|^2\leq\frac{4\gamma}{2-L\gamma}\cdot \big(f(x^0)-\min f\big)+\frac{4\gamma^2}{(2-L\gamma)^2}\sum_{t=0}^{k}\|\epsilon^t\|^2.
\end{equation}
\end{lemma}

\begin{proof}
~
Recalling the definition of $\Delta^k$ in \eqref{eq:def-delta} and noting $x^{k+1}_j=x_j^k$ for all $j\neq i_k$, we have:
\begin{align}\label{in-le2-temp-1}
    \langle \Delta^{k},\nabla f(x^k)\rangle=\left\langle x^{k+1}_{i_k}-x^{k}_{i_k},\nabla_{i_k} f(x^k)\right\rangle=-\frac{1}{\gamma}\|\Delta^k\|^2+\left\langle\epsilon^k, x^{k}_{i_k}-x^{k+1}_{i_k}\right\rangle,
\end{align}
where we have used the update rule in \eqref{inCD} to obtain the second equality.
By \eqref{eq:lip-ineq} and \eqref{in-le2-temp-1}, it holds that
\begin{align}
    f(x^{k+1})&\leq f(x^k)+\langle  \Delta^k, \nabla f(x^k)\rangle+\frac{L}{2}\|\Delta^k\|^2\nonumber\\
    &=f(x^k)+ \big(\frac{L}{2}-\frac{1}{\gamma}\big)\|\Delta^k\|^2+\langle\epsilon^k, x^{k}_{i_k}-x^{k+1}_{i_k}\rangle\label{in-le2-temp-2.1}\\
    &\overset{a)}{\leq} f(x^k)+\big(\frac{L}{4}-\frac{1}{2\gamma}\big)\|\Delta^k\|^2+\frac{\gamma\|\epsilon^k\|^2}{2-L\gamma},\label{in-le2-temp-2}
\end{align}
where $a)$ is from the Young's inequality
$\langle \epsilon^k, x^{k}_{i_k}-x^{k+1}_{i_k}\rangle\leq \frac{\gamma}{2-L\gamma}\|\epsilon^k\|^2+\frac{2-L\gamma}{4\gamma}\|\Delta^k\|^2$.
Summing \eqref{in-le2-temp-2}, rearranging terms, and noting $f(x^k)\ge \min f,\forall k$, we obtain the desired result and complete the proof.
\end{proof}

Also, we can bound partial gradient by the iterate change $\Delta^k$ and error term $\epsilon^k$ as follows.

\begin{lemma}\label{lem:bd-gik}
Assume \eqref{eq:lip-f}. Let $(x^k)_{k\geq 0}$ be generated by the inexact MC-BCD \eqref{inCD}. Then for $k\geq \uptau$, it holds
\begin{align}\label{yan}
    \|\nabla_{i_k}f(x^{k-\uptau+1})\|^2\leq 2L_r^2\cdot(\uptau-1)\cdot\sum_{d=k-\uptau+1}^{k-1}\|\Delta^d\|^2+\frac{4}{\gamma^2}\|\Delta^k\|^2+4\|\epsilon^k\|^2.
\end{align}
\end{lemma}

\begin{proof}
~
By the update rule in \eqref{inCD} and the definition of $\Delta^k$, we have $-\nabla_{i_k}f(x^{k})=\frac{\Delta^k_{i_k}}{\gamma}+\epsilon^k$.
Applying the triangle inequality to the above inequality yields
\begin{align}\|\nabla_{i_k}f(x^{k-\uptau+1})\|^2
\leq & ~ 2\|\nabla_{i_k}f(x^{k-\uptau+1})-\nabla_{i_k}f(x^{k})\|^2+2\left\|\frac{\Delta^k_{i_k}}{\gamma}+\epsilon^k\right\|^2\label{eq:bdpd-step1}\\
\leq & ~ 2\|\nabla_{i_k}f(x^{k-\uptau+1})-\nabla_{i_k}f(x^{k})\|^2+\frac{4}{\gamma^2}\|\Delta^k\|^2+4\|\epsilon^k\|^2.\nonumber
\end{align}
Note $\|\nabla_{i_k}f(x^{k-\uptau+1})-\nabla_{i_k}f(x^{k})\|^2\le \|\nabla f(x^{k-\uptau+1})-\nabla f(x^{k})\|^2$. Hence, it follows from the triangle inequality and the Lipschitz continuity of $\nabla f$ in \eqref{eq:lip-f} that
\begin{align*}
    \|\nabla_{i_k}f(x^{k-\uptau+1})\|^2&\leq 2\|\nabla f(x^{k-\uptau+1})-\nabla f(x^{k})\|^2+\frac{4}{\gamma^2}\|\Delta^k\|^2+4\|\epsilon^k\|^2\nonumber\\
    &\leq 2\cdot(\uptau-1)\cdot\sum_{d=k-\uptau+1}^{k-1}\|\nabla f(x^{d+1})-\nabla f(x^d)\|^2+\frac{4}{\gamma^2}\|\Delta^k\|^2+4\|\epsilon^k\|^2\nonumber\\
    &\leq 2L_r^2\cdot(\uptau-1)\cdot\sum_{d=k-\uptau+1}^{k-1}\|\Delta^d\|^2+\frac{4}{\gamma^2}\|\Delta^k\|^2+4\|\epsilon^k\|^2,
\end{align*}
which gives the desired result.
\end{proof}

\begin{remark}
If $\epsilon^k=0,\,\forall k$, then starting from \eqref{eq:bdpd-step1} and by the same arguments, we can have
\begin{align}\label{yan-noise0}
    \|\nabla_{i_k}f(x^{k-\uptau+1})\|^2\leq 2L_r^2\cdot(\uptau-1)\cdot\sum_{d=k-\uptau+1}^{k-1}\|\Delta^d\|^2+\frac{2}{\gamma^2}\|\Delta^k\|^2.
\end{align}
\end{remark}
Furthermore, we can lower bound full gradient by conditional partial gradient.

\begin{lemma}\label{lem:cbd-gik}
Let \eqref{mixt1} hold. Then it holds
\begin{align}\label{th1-t1}
    \EE\big(\|\nabla_{i_k}f(x^{k-\uptau+1})\|^2\mid\chi^{k-\uptau+1}\big){\geq} \frac{\pi^*_{\min}}{2}\|\nabla f(x^{k-\uptau+1})\|^2.
\end{align}
\end{lemma}

\begin{proof}
~
Taking conditional expectation, we have
$$\EE\big(\|\nabla_{i_k}f(x^{k-\uptau+1})\|^2\mid\chi^{k-\uptau+1}\big)= \sum_{j=1}^N  \|\nabla_j f(x^{k-\uptau+1})\|^2\cdot\mathbb{P}(i_k=j\mid\chi^{k-\uptau+1}).$$
By the Markov property, it holds $\mathbb{P}(i_k=j\mid\chi^{k-\uptau+1})=\mathbb{P}(i_k=j\mid i_{k-\uptau})=[\Phi(k-\uptau,\uptau-1)]_{i_{k-\uptau},j}$. Then the desired result is obtained from \eqref{th1-t0} and the fact $\sum_{i=1}^N \|\nabla_i f(\cdot)\|^2=\|\nabla f(\cdot)\|^2$.
\end{proof}

\begin{theorem}\label{th1}
Let Assumptions \ref{ass:mc} and \ref{ass:2} hold and $(x^k)_{k\geq 0}$ be generated by the inexact MC-BCD \eqref{inCD} with any constant stepsize $0<\gamma<\frac{2}{L}$.
We have the following results:
\begin{enumerate}
\item \textbf{Square summable noise:}
If the noise sequence satisfy $\sum_{k=0}^\infty\|\epsilon^k\|^2=\mathcal{E}<+\infty$. Then,
\begin{equation}\label{th1-result}
    \lim_{k\to\infty}\EE\|\nabla f(x^k)\|=0,
\end{equation}
and
\begin{equation}\label{th1-result-rate1}
    \EE\left[\min_{1\leq t\leq k}\|\nabla f(x^t)\|^2\right]\leq\frac{2}{(k+1)\pi_{\min}^*}\left[C_1(\uptau)\cdot \big(f(x^0)-\min f\big)+\big(C_2(\uptau)+4\big)\mathcal{E}\right].
\end{equation}


\item \textbf{Non-square-summable noise:} If $\|\epsilon^k\|^2\leq S,\, \forall k\ge0$ for some positive number $S>0$, then
\begin{align}\label{th1-result-rate2}
    \EE\left[\min_{1\leq t\leq k}\|\nabla f(x^{t})\|^2\right]\leq\frac{2}{(k+1)\pi_{\min}^*}C_1(\uptau)\cdot \big(f(x^0)-\min f\big)+\frac{2}{\pi_{\min}^*}\left(\frac{C_2(\uptau)(k+\uptau)}{k+1}+4\right)S.
\end{align}
\end{enumerate}
The constants used above are
\begin{equation}\label{eq:C-J-E}
C_1(\uptau): = \frac{4\gamma}{2-L\gamma}\left(2L_r^2(\uptau-1)^2+\frac{4}{\gamma^2}\right),\quad C_2(\uptau): = \frac{4\gamma^2}{(2-L\gamma)^2}\left(2L_r^2(\uptau-1)^2+\frac{4}{\gamma^2}\right).
\end{equation}
\end{theorem}
\begin{proof}
%
%
~
In the case of square summable noise, we have $\epsilon^k\to \textbf{0}$ as $k\to\infty$. In addition, it follows from \eqref{eq:sum-dx} that $\sum_{k=0}^\infty\|\Delta^k\|^2<+\infty$ and thus $\Delta^k\to \textbf{0}$ as $k\to\infty$. Hence, \eqref{yan} implies
\begin{align}\label{yanling}
    \lim_{k\to\infty}\|\nabla_{i_k}f(x^{k-\uptau+1})\|^2=0.
\end{align}
Taking expectation on \eqref{yanling} and using the Lebesgue dominated convergence theorem, we have $$\lim_{k\to\infty}\EE\|\nabla_{i_k}f(x^{k-\uptau+1})\|^2=0.$$
Hence from \eqref{th1-t1}, it follows that
$$\lim_{k\to\infty}\EE\|\nabla f(x^{k})\|^2=\lim_{k\to\infty}\EE\|\nabla f(x^{k-\uptau+1})\|^2\leq\frac{2}{\pi^*_{\min}}\lim_{k\to\infty} \EE\|\nabla_{i_k}f(x^{k-\uptau+1})\|^2=0,$$
and thus \eqref{th1-result} holds by the Jensen's inequality $(\EE\|\nabla f(x^k)\|)^2\leq\EE\|\nabla f(x^k)\|^2$.

Note $\sum_{t=\uptau-1}^{k}\sum_{d=t-\uptau+1}^{t-1}\|\Delta^d\|^2\le (\uptau-1)\sum_{d=0}^{k-1}\|\Delta^d\|^2$ for any $k\geq\uptau$. Therefore, summing both sides of \eqref{yan} yields
\begin{align}
\sum_{t=\uptau-1 }^k\|\nabla_{i_t}f(x^{t-\uptau+1})\|^2 \leq & 2L_r^2(\uptau-1)^2\sum_{d=0}^{k-1}\|\Delta^d\|^2+\frac{4}{\gamma^2}\sum_{t=\uptau-1 }^k\|\Delta^t\|^2+4\sum_{t=\uptau-1 }^k\|\epsilon^t\|^2\nonumber\\
\leq & \left(2L_r^2(\uptau-1)^2+\frac{4}{\gamma^2}\right)\sum_{t=0}^k\|\Delta^t\|^2+4\sum_{t=\uptau-1 }^k\|\epsilon^t\|^2.\label{eq:bd-pg2}
\end{align}
The inequality in \eqref{eq:bd-pg2} together with \eqref{eq:sum-dx} and the assumption on $\epsilon^k$ gives
\begin{align}\label{jiahe-1}
\sum_{t=\uptau-1 }^\infty\|\nabla_{i_t}f(x^{t-\uptau+1})\|^2\le C_1(\uptau)\cdot \big(f(x^0)-\min f\big)+\big(C_2(\uptau)+4\big)\mathcal{E},
\end{align}
where $C_1(\uptau)$ and $C_2(\uptau)$ are defined in \eqref{eq:C-J-E}.
In addition, we have
\begin{equation}\label{eq:bd-gd}
(k+1)\cdot\EE\left[\min_{0\leq t\leq k}\|\nabla f(x^{t})\|^2\right]\leq\sum_{t=0}^k \EE\|\nabla f(x^{t})\|^2=\sum_{t=\uptau-1}^{k+\uptau-1} \EE\|\nabla f(x^{t-\uptau+1})\|^2\le \frac{2}{\pi_{\min}^*}\sum_{t=\uptau-1}^{k+\uptau-1}\EE\|\nabla_{i_t}f(x^{t-\uptau+1})\|^2,
\end{equation}
where the last inequality follows from \eqref{th1-t1}. Now the result in \eqref{th1-result-rate1} is obtained from the above inequality together with that in \eqref{jiahe-1}.


In the case of $\|\epsilon^k\|^2\leq S,\, \forall k\ge0$, we have from \eqref{eq:sum-dx} and \eqref{eq:bd-pg2} that
\begin{equation*}
\sum_{t=\uptau-1 }^k\|\nabla_{i_t}f(x^{t-\uptau+1})\|^2 \le \left(2L_r^2(\uptau-1)^2+\frac{4}{\gamma^2}\right)\left(\frac{4\gamma}{2-L\gamma}\cdot \big(f(x^0)-\min f\big)+\frac{4\gamma^2(k+1)S}{(2-L\gamma)^2}\right) + 4(k-\uptau+2)S.
\end{equation*}
In the above inequality, setting $k$ to $k+\uptau-1$ and using \eqref{eq:bd-gd} give
the result in \eqref{th1-result-rate2}.
\end{proof}

Although MC-BCD has sample bias, we can still use a constant stepsize. In fact, Theorem \ref{th1} indicates the stepsize can be as large as traditional BCD. The assumption on the noise sequence is weaker than the commonly found assumption $\sum_{k}\|\epsilon_k\|<+\infty$.
When the noise sequence is non-diminishing, we have a final error that approximately matches the noise level. This is useful in an application in Sec. 5.3, where computing $\nabla_{i_k} f$ may involve certain sampling that becomes too expensive to require asymptotically vanishing noise.  

\subsection{Convergence rates for convex minimization}
When $f$ is convex, we can estimate the rates of expected objective error. 
We let
\begin{equation}\label{ft}
    F_t:=\EE f(x^{t\cdot\uptau})- \min f\quad\text{and}\quad \overline{x} =\mathrm{Proj}_{\argmin f}(x).
\end{equation}
First, we present an important technical lemma, which will be used to derive both sublinear and linear convergence results.
\begin{lemma}\label{lemsub}
Let $(x^k)_{k\geq 0}$ be generated by MC-BCD \eqref{CD} with  $0<\gamma<\frac{2}{L}$. When $f$  is convex, we have
\begin{equation}\label{lemsub-t6}
    F_t^2\leq C_{\uptau}\cdot(F_t-F_{t+1})\cdot\EE\|x^{t\cdot\uptau}-\overline{x^{t\cdot\uptau}}\|^2,
\end{equation}
where the constant is
\begin{equation}\label{aboutc}
    C_{\uptau}:=\frac{\max\{4L_r^2\cdot(\uptau-1),\frac{4}{\gamma^2}\}}{(\frac{1}{\gamma}-\frac{L}{2})\cdot\pi^*_{\min}}.
\end{equation}
\end{lemma}

\begin{proof}
~
Since $\epsilon^k=0,\,\forall k$, taking expectations of both sides of (\ref{yan-noise0}) and using (\ref{th1-t1}) yield
\begin{eqnarray}\label{lemsub-t1-0}
    \EE\|\nabla f(x^{k-\uptau+1})\|^2\leq \frac{\max\left\{4L_r^2\cdot(\uptau-1),\frac{4}{\gamma^2}\right\}}{\pi^*_{\min}}\cdot\sum_{d=k-\uptau+1}^{k}\EE\|\Delta^d\|^2,
\end{eqnarray}
For each $d$, we have from (\ref{in-le2-temp-2.1}) with $\epsilon^k=0$ that
\begin{eqnarray}\label{lemsub-t2-0}
    \EE\|\Delta^d\|^2\leq\frac{\EE f(x^d)-\EE f(x^{d+1})}{\frac{1}{\gamma}-\frac{L}{2}}.
\end{eqnarray}
Substituting (\ref{lemsub-t2-0}) into (\ref{lemsub-t1-0}) and recalling the definition of $C_{\uptau}$ in \eqref{aboutc} give
\begin{eqnarray}\label{lemsub-t2}
    \EE\|\nabla f(x^{k-\uptau+1})\|^2\leq C_{\uptau}\left[\EE f(x^{k-\uptau+1})-\EE f(x^{k+1})\right].
\end{eqnarray}
For any integer $t$, letting $k=(t+1)\cdot\uptau-1$ in (\ref{lemsub-t2}), we have
\begin{eqnarray}\label{lemsub-t3}
    \EE\| \nabla f(x^{t\cdot\uptau})\|^2\leq C_{\uptau}\left[F_t-F_{t+1}\right].
\end{eqnarray}
On the other hand, it follows from convexity of $f$ that
\begin{eqnarray}\label{lemsub-t4}
    F_t=\EE f(x^{t\cdot\uptau})-\min f\leq \EE\left\langle\nabla f(x^{t\cdot\uptau}), x^{t\cdot\uptau}-\overline{x^{t\cdot\uptau}}\right\rangle.
\end{eqnarray}
Now  square both sides of (\ref{lemsub-t4}) and apply the Cauchy-Schwarz inequality to have
\begin{eqnarray}\label{lemsub-t5}
    F_t^2\leq \EE\|\nabla f(x^{t\cdot\uptau})\|^2\cdot\EE\|x^{t\cdot\uptau}-\overline{x^{t\cdot\uptau}}\|^2.
\end{eqnarray}
Substituting (\ref{lemsub-t3}) into (\ref{lemsub-t5}) yields \eqref{lemsub-t6}, and we complete the proof.
\end{proof}

\subsubsection{Sublinear convergence rate}
A well-known result in convergence analysis is that a nonnegative sequence $(a_k)_{k\ge0}$ that obeying $a_{k+1}\leq a_k$ and $a_{k+1}\le a_k-\eta a_k^2$, for some $\eta>0$ and all $k\geq 0$ satisfies
\begin{align}\label{eq:bd-ak}
a_k\le \frac{a_0}{a_0\eta k+1}.
\end{align}
It can be proved by observing
$\frac{1}{a_{k+1}}-\frac{1}{a_{k}} \geq \eta.$
\begin{theorem}\label{th2}
Under Assumptions \ref{ass:mc} and \ref{ass:2}, let $(x^k)_{k\geq 0}$ be generated by  MC-BCD \eqref{CD} with  $0<\gamma<\frac{2}{L}$. Assume that $f$ is convex and the level set $\mathcal{X}_0=\{x\in\RR^N: f(x)\le f(x^0)\}$ is bounded with diameter $R=\max_{x,y\in\mathcal{X}_0}\|x-y\|$. Then we have
\begin{equation}
    \EE f(x^k)-\min f \leq \frac{F_0C_{\uptau}R^2}{F_0 \lfloor \frac{k}{\uptau}\rfloor+C_{\uptau}R^2}, 
\end{equation}
where $C_{\uptau}$ is the constant defined in \eqref{aboutc}, and $\uptau$ is the $\frac{\pi_{\min}^*}{2}$-mixing time defined in \eqref{mixt1}. 
\end{theorem}

\begin{proof}
~
From \eqref{in-le2-temp-2.1} with $\epsilon^k=0,\,\forall k$ and $0<\gamma<\frac{2}{L}$, it follows that  $f(x^k)$ is monotonically nonincreasing about $k$, and thus $x^k\in\mathcal{X}_0$ for all $k$. Therefore, $\|x^{t\cdot\uptau}-\overline{x^{t\cdot\uptau}}\|^2\le R^2, \,\forall t$.
Substituting this inequality into \eqref{lemsub-t6} gives
$    F_t^2\leq C_{\uptau}R^2\cdot(F_t-F_{t+1})$,
or equivalently $F_{t+1}\le F_t-\frac{F_t^2}{C_{\uptau}R^2}$. 
From \eqref{eq:bd-ak} we obtain
$$F_t\le \frac{F_0}{\frac{F_0 t}{C_{\uptau}R^2}+1},\,\forall t\ge0.$$
Since $f(x^k)$ is nonincreasing about $k$, it follws that 
\begin{equation*}
    \EE f(x^k)-\min f\leq F_{\lfloor \frac{k}{\uptau}\rfloor}\leq \frac{F_0}{\frac{F_0 \lfloor \frac{k}{\uptau}\rfloor}{C_{\uptau}R^2}+1}=\frac{F_0C_{\uptau}R^2}{F_0 \lfloor \frac{k}{\uptau}\rfloor+C_{\uptau}R^2},
\end{equation*}
which completes the proof.
\end{proof}

\begin{remark}
We consider a standard stepsize  $\gamma=\frac{1}{L}$ and compare random BCD and MC-BCD.
In [Theorem 1, \cite{nesterov2012efficiency}], it is shown that random BCD has the rate $
    \EE f(x^k)-\min f=O(\frac{N\cdot R^2\cdot L}{ k}).$
We stress that, with our notation, $\nabla f$ is $(N\cdot L)$-Lipschitz in the worst case. When our Markov chain uses a complete graph, we can have a uniform stationary distribution and $\uptau=1$. In this case, MC-BCD reduces to random BCD, and our complexity of MC-BCD is also
$
\EE f(x^k)-\min f=O( \frac{N\cdot R^2\cdot L}{ k}).
$
In this sense, we have generalized random BCD with a matching complexity. If the Markov chain promises a geometric mixing rate, i.e., $\uptau=O(\ln N)$, then our convergence rate result becomes $
    \EE f(x^k)-\min f=O(\frac{ N\cdot \ln^2 N\cdot R^2\cdot L}{k})$.
While in cyclic BCD, we have $f(x^k)-\min f=O(\frac{N^2\cdot R^2\cdot L}{ k})$ from  \cite[Corollary 3.8]{BeckTetruashvili2013_convergence}.\footnote{The authors in \cite[Corollary 3.8]{BeckTetruashvili2013_convergence} presents this results in the perspective of epochs, while here we present the rate in the perspective of iterations. Thus, their result is multiplied by $N$ for comparison.}
That is,, in terms of worst-case guarantee,  MC-BCD performs slightly worse than i.i.d. random BCD but better than cyclic BCD.
\end{remark}

\subsubsection{Linear convergence rate.}

To have linear convergence, we consider the restricted $\nu$-strongly convex function:
\begin{equation}\label{qg}
    f(x)-\min f\geq\nu\|x-\overline{x}\|^2,\quad \text{for all}~x\in\RR^N,~\overline{x} =\mathrm{Proj}_{\argmin f}(x).
\end{equation}
\begin{theorem}\label{th-li}
Under Assumptions \ref{ass:mc} and \ref{ass:2}, let $(x^k)_{k\geq 0}$ be generated by  MC-BCD \eqref{CD} with  $0<\gamma<\frac{2}{L}$.
If $f$ satisfies condition (\ref{qg}), then
\begin{equation}
    \EE f(x^k)-\min f \leq F_0\left(1-\frac{\nu}{C_{\uptau}}\right)^{\lfloor \frac{k}{\uptau}\rfloor}. 
\end{equation}
\end{theorem}
\begin{proof}
~
Immediately from \eqref{qg}, we have the bound
\begin{equation}\label{th-li-t1}
    \EE\|x^{t\cdot\uptau}-\overline{x^{t\cdot\uptau}}\|^2\leq\frac{\EE f(x^{t\cdot\uptau})-\min f}{\nu}=\frac{F_{t}}{\nu}.
\end{equation}
Substituting the above inequality into \eqref{lemsub-t6} yields
 $   F_t^2\leq \frac{C_{\uptau}}{\nu}\cdot(F_t-F_{t+1})\cdot F_t$,
or equivalently $F_{t+1}\le (1-\frac{\nu}{C_{\uptau}})F_t.$
Hence,
$$F_t\le F_0\left(1-\frac{\nu}{C_{\uptau}}\right)^t ,\,\forall t\ge0.$$
Again from monotonicity of $f(x^k)$ about $k$, it follows that
\begin{equation}
    \EE f(x^k)-\min f \leq F_{\lfloor \frac{k}{\uptau}\rfloor} \le F_0\left(1-\frac{\nu}{C_{\uptau}}\right)^{\lfloor \frac{k}{\uptau}\rfloor},
\end{equation}
which completes the proof.
\end{proof}

\begin{remark}
If we consider the stepsize $\gamma=\frac{1}{L}$ and assume the Markov chain enjoys a uniform stationary distribution, then we get the rate
$
\EE(f(x^k)-\min f)=O\left(\left(1-\frac{\nu}{N\cdot\max\{8\kappa L\cdot(\uptau-1),8L\}}\right)^{\lfloor \frac{k}{\uptau}\rfloor}\right).
$
\end{remark}

\section{Extension to nonsmooth problems}
All results established in previous sections assume the smoothness of the objective function. In this section, we add separable, possibly nonsmooth functions to the objective: 
\begin{equation}\label{model-com}
    \Min~~F(x)\equiv f(x_1,x_2,\ldots,x_N)+\sum_{i=1}^N g_i(x_i).
 \end{equation}
Here, $f:\mathbb{R}^N\mapsto \mathbb{R}$ is a differentiable function, $\nabla_i f$ is Lipschitz continuous for each $i=1,2,\dots,N$, and $g_i:\mathbb{R}\mapsto \mathbb{R}$ is a closed proper function. Note that we do not assume convexity on either $f$ or $g_i$'s. Toward finding a solution to \eqref{model-com}, we propose the inexact Markov chain  proximal block coordinate descent (iMC-PBCD).

 Given a graph $\mathcal{G}=(\mathcal{V},\mathcal{E})$, the iMC-PBCD iteratively performs:
\begin{subequations}\label{MC-PBCD}
\begin{align}
 \text{sample}~& i_{k}\in\{j:(i_{k-1},j)\in \mathcal{E}\} \sim P_{i_{k-1},j}(k),\\
\label{ProxCD}
\text{compute}~& x^{k+1}_{j} =\left\{
\begin{array}{ll}\prox_{\gamma g_j}\left( x^k_{j} - \gamma \big(\nabla_{j} f(x^{k})+\epsilon^k\big)\right), &\text{ if }j=i_k,\\[0.1cm]
x_j^k,&\text{ if }j\neq i_k.
\end{array}
\right.
\end{align}
\end{subequations}
In the above update, $\gamma$ is a step size,
$\epsilon^k$ denotes the error in evaluating the partial gradient, and $\prox_\psi(y)$ is the proximal mapping of a closed function $\psi$ at $y$, defined as
$$\prox_\psi(y)\in\argmin_x \left\{\psi(x)+\frac{1}{2}\|x-y\|^2\right\}.$$ 

To characterize the property of a solution, we employ the notion of subdifferential \cite[Definition 8.3]{rockafellar2009variational}. 
\begin{definition}[subdifferential] Let  $J: \mathbb{R}^N \rightarrow (-\infty, +\infty]$ be a proper and lower semicontinuous function.
\begin{enumerate}
  \item For any $x\in \textrm{dom} (J)$, the Fr$\acute{e}$chet subdifferential of $J$ at $x$, denoted as $\hat{\partial}J (x)$, is the set of all vectors $u\in \mathbb{R}^N$ that satisfies
  $$\lim_{y\neq x}\inf_{y\rightarrow x}\frac{J(y)-J(x)-\langle u, y-x\rangle}{\|y-x\|}\geq 0.$$
If $x\notin \textrm{dom} (J)$, then $\hat{\partial}J(x)=\emptyset$.

\item The limiting subdifferential, or simply the subdifferential, of $J$ at $x\in \textrm{dom}(J)$, denoted as $\partial J(x)$, is defined as 
\begin{align}
\partial J(x):=\{u\in\mathbb{R}^N: \exists~(x^k)_{k\geq 0} ~\text{and}~ u^k\in \hat{\partial}J(x^k)\text{ such that }J(x^k)\rightarrow J(x)\text{ and }u^k\rightarrow u~\emph{as}~k\rightarrow \infty\}\nonumber.
\end{align}
\end{enumerate}
\end{definition}
%
%
%
The first-order optimality condition for $x$ to be a solution of \eqref{model-com} is
\begin{align}\label{critical}
    \textbf{0}\in \partial F(x).
\end{align}
Any such point is called a critical point of $F$.

{The proofs below are quite different from previous ones because we cannot bound the gradient with $\|\Delta^k\|$ any more, i.e., the core relation \eqref{yan} fails to hold. Consequently, the convergence result in this section is new. Also, we cannot specify the convergence rates yet.}
\begin{lemma}
Under Assumption \ref{ass:2}, let $(x^k)_{k\geq 0}$ be generated by iMC-PBCD \eqref{MC-PBCD} with  $0<\gamma<\frac{1}{L}$. If $\sum_{k=0}^\infty\|\epsilon^k\|^2<\infty$, then
\begin{align}\label{bounded-descent-temp-4}
\lim_{k\to\infty}\Delta^k=\textbf{0},
\end{align}
where $\Delta^k$ is defined in \eqref{eq:def-delta}.
\end{lemma}

\begin{proof}
By the definition of the proximal mapping, the update in \eqref{ProxCD} can be equivalently written as
\begin{equation}\label{eq:equiv-x-update}
x^{k+1}_{i_k}\in \argmin_{x_{i_k}}\left\{ \left\langle x_{i_k}-x_{i_k}^k, \big(\nabla_{i_k} f(x^{k})+\epsilon^k\big)\right\rangle + \frac{1}{2\gamma}\|x_{i_k}-x_{i_k}^k\|^2+g_{i_k}(x_{i_k})\right\}.
\end{equation}
Therefore,
\begin{align}\label{variat}
 \left\langle x^{k+1}_{i_k}-x^{k}_{i_k}, \nabla_{i_k} f(x^{k})+\epsilon^k\right\rangle+\frac{1}{2\gamma}\|x^{k+1}_{i_k}-x^k_{i_k}\|^2+g_{i_k}(x^{k+1}_{i_k})\leq g_{i_k}(x^{k}_{i_k}).
\end{align}
By the Young's inequality and the definition of $\Delta^k$ in \eqref{eq:def-delta}, it holds that
$$\left\langle x^{k+1}_{i_k}-x^{k}_{i_k},\epsilon^k\right\rangle\leq \frac{1}{4}\left(\frac{1}{\gamma}-L\right)\|\Delta^k\|^2+\frac{\|\epsilon^k\|^2}{\frac{1}{\gamma}-L}.$$
In addition, it follows from \eqref{eq:lip-ineq} that
$$f(x^{k+1})\le f(x^k)+\big\langle \Delta^k,\nabla f(x^k)\big\rangle+\frac{L}{2}\|\Delta^k\|^2.$$
Adding the above two inequalities into \eqref{variat} and recalling the definition of $\Delta^k$ in \eqref{eq:def-delta} give
\begin{align*}
f(x^{k+1})+g_{i_k}(x^{k+1}_{i_k})+\frac{1}{2\gamma}\|\Delta^k\|^2
\leq  f(x^k)+g_{i_k}(x^{k}_{i_k})+\frac{1}{4}\left(\frac{1}{\gamma}-L\right)\|\Delta^k\|^2+\frac{\|\epsilon^k\|^2}{\frac{1}{\gamma}-L}+\frac{L}{2}\|\Delta^k\|^2.
\end{align*}
Rearranging terms of the above inequality and noting $g_j(x_j^{k+1})=g_j(x_j^k)$ for all $j\neq i_k$, we have
\begin{align*}
    F(x^{k+1})+\left(\frac{1}{4\gamma}-\frac{L}{4}\right)\|\Delta^k\|^2\leq F(x^k)+\frac{\|\epsilon^k\|^2}{\frac{1}{\gamma}-L},
\end{align*}
or equivalently
\begin{align}\label{bounded-descent-temp-2}
\frac{1}{4}\left(\frac{1}{\gamma}-L\right)\|\Delta^k\|^2\le F(x^k) - F(x^{k+1})+\frac{\|\epsilon^k\|^2}{\frac{1}{\gamma}-L}.
\end{align}
Summing up the above inequality over $k$, using the conditions $0<\gamma<\frac{1}{L}$ and $\sum_{k=0}^\infty\|\epsilon^k\|^2<\infty$, and also noting $F$ is lower bounded yield $\sum_{k=0}^\infty\|\Delta^k\|^2<\infty$, which implies \eqref{bounded-descent-temp-4} and completes the proof.
\end{proof}


\begin{theorem}
Under Assumptions \ref{ass:mc} and \ref{ass:2}, let $(x^k)_{k\geq 0}$ be generated by iMC-PBCD \eqref{MC-PBCD} with  $0<\gamma<\frac{1}{L}$. If $\sum_{k=0}^\infty\|\epsilon^k\|^2<\infty$, then any cluster point  of $(x^k)_{k\geq 0}$ is a critical point of $F$ almost surely. 
\end{theorem}
\begin{proof}
%
%
%
%
%
%
%
%
By the first optimality condition of \eqref{eq:equiv-x-update}, it holds
\begin{align*}
\frac{-\Delta^{k}_{i_k}}{\gamma}-\nabla_{i_k}f(x^k)-\epsilon^k\in \partial g_{i_k}(x^{k+1}_{i_k}),
\end{align*}
or equivalently
\begin{align}\label{bounded-descent-temp-5}
-\frac{\Delta^{k}_{i_k}}{\gamma}+\nabla_{i_k}f(x^{k+1})-\nabla_{i_k}f(x^k)-\epsilon^k\in \nabla_{i_k}f(x^{k+1})+\partial g_{i_k}(x^{k+1}_{i_k})= \partial_{i_k}F(x^{k+1}).
\end{align}
From \eqref{bounded-descent-temp-4} and also the Lipschitz continuity of $\nabla_i f$, we have from \eqref{bounded-descent-temp-5} that
\begin{align}\label{distpar}
    \lim_{k\to\infty}\textrm{dist}\left(0,\partial_{i_k}F(x^{k+1})\right)\leq\lim_{k\to\infty}\left\|-\frac{\Delta^{k}_{i_k}}{\gamma}+\nabla_{i_k}f(x^{k+1})-\nabla_{i_k}f(x^k)-\epsilon^k\right\|=0.
\end{align}

Let $\bar{x}$ be a cluster point of $(x^k)_{k\geq 0}$ and thus there is a subsequence $(x^k)_{k\in\mathcal{K}}\rightarrow \bar{x}$. If necessary, taking a subsubsequence, we can assume $|k_1-k_2|\geq \mathcal{J}$ for any $k_1,k_2\in\mathcal{K}$. We go to prove the following claim:
\begin{equation}\label{aseq}
\text{For any }j\in [N],\text{ there are infinite }k\in\mathcal{K}\text{ such that }
i_{k}=j,\,~a.s.
\end{equation}

If the above claim is not true, then for some $j\in[N]$, with nontrivial probability, there are only finite $k\in\mathcal{K}$ such that $i_k=j$. Dropping these finitely many $k$'s in $\mathcal{K}$, we obtain a new subsequence $\mathcal{\hat{K}}=\{k_1,k_2,\ldots\}$ and $i_k\neq j$ for any $k\in\hat{\mathcal{K}}$. By the Markov property, it holds that for any $m\ge 1$,
 \begin{align}\label{infty-p}
     &\PP(i_{k_1}\neq j, i_{k_2}\neq j, i_{k_3}\neq j,\ldots, i_{k_m}\neq j)\nonumber\\
     =&\PP(i_{k_1}\neq j)\PP(i_{k_2}\neq j\mid i_{k_1}\neq j)\PP(i_{k_3}\neq j\mid i_{k_2}\neq j)\ldots\PP(i_{k_m}\neq j\mid i_{k_{m-1}}\neq j).
 \end{align}
For any $k_{t-1}, k_t\in\hat{\mathcal{K}}$, since $k_t-k_{t-1}\ge\uptau$, then we have from \eqref{th1-t0} that
$\PP(i_{k_t}=j\mid i_{k_{t-1}}\neq j)\ge\frac{\pi_{\min}^*}{2}.$
Hence
\begin{align}\label{cha}
   \PP(i_{k_t}\neq j\mid i_{k_{t-1}}\neq j)=1-  \PP(i_{k_t}=j\mid i_{k_{t-1}}\neq j) \le 1-\frac{\pi_{\min}^*}{2},
\end{align}
and thus it follows from \eqref{infty-p} that
%
  \begin{align}
     \PP(i_{k_1}\neq j, i_{k_2}\neq j, i_{k_3}\neq j,\ldots, i_{k_m}\neq j)\leq \left(1-\frac{\pi_{\min}^*}{2}\right)^{m-1}.
 \end{align}
 Letting $m\to\infty$, we conclude that
 $$\PP\big(\mathcal{K}\text{ only contains finitely many }k\text{ such that }i_k=j\big)=0,$$
 and thus the claim in \eqref{aseq} is true.

Now for any $j\in [N]$, taking $k\in\mathcal{K}$ such that $i_k=j$ and letting $k\to\infty$, we have from the fact  $(x^{k+1})_{k\in\mathcal{K}}\rightarrow \bar{x}$ because of \eqref{bounded-descent-temp-4} and also the outer-continuity of subdifferential that
$$\textrm{dist}(0,\partial_{j}F(\bar{x}))=\lim_{k\in\mathcal{{K}}, i_k=j}\textrm{dist}(0,\partial_{i_k}F(x^{k+1}))=0,\, a.s.$$
Therefore, we complete the proof.
\end{proof}

\section{Empirical Markov chain dual coordinate ascent}
In this section, we consider a special case of the risk minimization problem in form of \eqref{regerm}. As we mentioned in section 1.1, if it is easy to get i.i.d. samples from the distribution $\Pi$ of the the sample space, then we can easily apply SDCA to \eqref{regerm}. However, there are some cases where the distribution $\Pi$ is not explicitly given and the samples are generated by a simulator, such as an MCMC sampler. Assume that the samples generated by the simulator forms an Markov chain with stationary distribution $\Pi$. Generating i.i.d. samples may take very long time in this case, instead we want to make use of all the samples on a sample trajectory, which are not i.i.d. distributed.

Assume that the sample space $\Xi$ is finite. Let $p_{\xi}\in (0,1)$ denote the probability mass of $\xi\in\Xi$. Then, problem \eqref{regerm} can be presented as
\begin{align}\label{regerm2}
    \Min_{w\in\mathbb{R}^n} ~~\sum_{\xi\in\Xi}p_{\xi}F(w^{\top}\xi)+\frac{\lambda}{2}\|w\|^2.
\end{align}
The objective function involves unknown parameters $(p_{\xi})_{\xi\in\Xi}$. One way to solve this problem is to do the following two steps: first run the simulator for long enough time to get an estimation of $(p_{\xi})_{\xi\in\Xi}$ (e.g. use frequency), denoted by $(\bar{p}_{\xi})_{\xi\in\Xi}$; then minimize \eqref{regerm} with $(\bar{p}_{\xi})_{\xi\in\Xi}$ by SDCA. 
The SDCA iteration in this case would be:
\begin{align}
v^k&=v^{k-1}+\frac{\alpha^k_{\xi^k} \xi^k}{\lambda }-\frac{\alpha^{k-1}_{\xi^k} \xi^k}{\lambda },\label{sdca-erm-1}\\
    \alpha^{k+1}_{\xi^k}&=\alpha^k_{\xi^k}-\gamma\Big((\xi^k)^{\top} v^k-\nabla F^*\big(\frac{-\alpha_{\xi^k}^k}{\bar{p}_{\xi^k}}\big)\Big),\label{sdca-erm-2}
\end{align}
where $\xi^k$ is uniformly randomly chosen from $\Pi$, $F^*$ is the conjugate function of $F$, $\alpha:=(\alpha_{\xi})_{\xi\in\Xi}$ are dual variables, and $\gamma$ is the stepsize.

Compared with SDCA, the advantage of MC-DCA is to do sampling and minimization simultaneously. However, it still needs to estimate $(p_{\xi})_{\xi\in\Xi}$. To address this issue, we introduce a practical way that approximates $p_{\xi}$ by keeping a count $c_{\xi}(k)$, the times that sample $\xi$ is chosen between iterations $1$ and $k$. We estimate $p_{\xi}$ by the sample frequency $c_{\xi}(k)/k$. We call it \emph{empirical MC-DCA}. The empirical MC-DCA iteration is almost the same as SDCA iteration except that $(\xi^k)_{k\geq 0}\subseteq\Xi$ is a Markov chain and $\bar{p}_{\xi^k}=c_{\xi}(k)/k$.


We provide the theoretical performance of the empirical MC-DCA under an lower bounded assumption on the frequency and geometric convergence of the Markov chain sampling.
\begin{assumption}\label{lower}
There exists  universal constant $\delta>0$ such that
for any $\xi\in\Pi$ and integer $k$, $c_{\xi}(k) / k\geq \delta>0$. There exists $0<\lambda<1$ such that $|\PP(\xi^{k}=\xi)-\PP(\hat{\xi}=\xi)|=O(\lambda^k)$, $\xi\in\Xi$ and $\hat{\xi}\in\Pi$.
\end{assumption}
The time-homogeneous, irreducible, and aperiodic Markov Chain can satisfy Assumption \ref{lower}.

\begin{corollary}\label{empi-mc-dca}
Let $\alpha^k:=(\alpha^k_{\xi})_{\xi\in\Xi}$ be generated by the empirical  MC-DCA and Assumption \ref{lower} hold. Then for the dual function given in \eqref{dual}, by denoting $A:=\sup_{1\leq i\leq k,\xi\in\Xi}\{\|\alpha^k_{\xi}\|^2\}$, it holds that
\begin{align}
    \EE\big[\min_{1\leq i\leq k}\|\nabla D(\alpha^{i})\|^2\big]=O(\frac{A\cdot\ln^2 k}{k}),
\end{align}
\end{corollary}

\begin{proof}
Obviously, the empirical MC-DCA can be regarded as the inexact MC-BCD to minimize $D(\alpha)$ with the noise
\begin{align}\label{empirical-noise}
    e^k=\nabla F^*\big(\frac{-\alpha_{\xi^k}^k}{p_{\xi}}\big)-\nabla F^*\big(\frac{-\alpha_{\xi^k}^k}{c_{\xi}(k)/k}\big).
\end{align}
We have presented the convergence result of the inexact MC-BCD in Theorem \ref{th1}. Thus, our work turns to bounding $e^k$. With Assumption \ref{lower},
\begin{align}
   \|e^k\|^2=O\big(A\cdot\frac{\|c_{\xi}(k)-k\cdot p_{\xi}\|^2}{k^2}\big).
\end{align}
We now estimate the upper bound of $\EE\|c_{\xi}(k)-k\cdot p_{\xi}\|^2$. Denote $\textbf{1}_{\xi}(\cdot)$ as the variable valued as $1$ when $\cdot=\xi$ and 0 when $\cdot$ being others. Then, $c_{\xi}(k)$ can be represented as  $$c_{\xi}(k)=\sum_{i=1}^k \textbf{1}_{\xi}(\xi^i).$$
Direct calculation then gives
\begin{align*}\label{uniformbetter-t0}
    \EE\|c_{\xi}(k)-k\cdot p_{\xi}\|^2=   \EE\|\sum_{i=1}^k \textbf{1}_{\xi}(\xi^i)-k\cdot p_{\xi}\|^2=\underbrace{\sum_{i=1}^k\EE\textbf{1}_{\xi}^2(\xi^i)}_{a)}\underbrace{-2kp_{\xi}\sum_{i=1}^k\EE\textbf{1}_{\xi}(\xi^i)}_{b)}+\underbrace{2\sum_{i<j}\EE\big(\textbf{1}_{\xi}(\xi^i)\textbf{1}_{\xi}(\xi^j)\big)}_{c)}+k^2p_{\xi}^2.
\end{align*}
With Assumption \ref{lower}, we have
\begin{align}
    a)=kp_{\xi}+O(\sum_{i=1}^k\lambda^i)=kp_{\xi}+O(\frac{1}{1-\lambda}).
\end{align}
Similarly, we can derive
\begin{align}
    b)=-2k^2p_{\xi}^2+O(\frac{1}{1-\lambda}).
\end{align}
Now, we focus on bounding $c)$. The difficulty is the dependence of the variables. Denote the $\sigma$-algebra $\chi^k$ generated by $\xi^0,\xi^1,\ldots,\xi^k$, i.e.,
$\chi^k:=\sigma(\xi^0,\xi^1,\ldots,\xi^k)$.
Thus, we first derive the conditional expectation and then use the property $\EE(\EE(\cdot\mid\chi^i))=\EE(\cdot)$. Noting that $i<j$, we have
\begin{align}
    \EE\big(\textbf{1}_{\xi}(\xi^i)\textbf{1}_{\xi}(\xi^j)\mid\chi^i\big)=\PP(\xi^{j}=\xi\mid\xi^{i}=\xi)\cdot\EE\big(\textbf{1}_{\xi}(\xi^i)\mid\chi^i\big).
\end{align}
Taking expectations on both sides, we are then led to
\begin{align}
    \EE\big(\textbf{1}_{\xi}(\xi^i)\textbf{1}_{\xi}(\xi^j)\big)=\PP(\xi^{j}=\xi\mid\xi^{i}=\xi)\cdot\EE\big(\textbf{1}_{\xi}(\xi^i)\big)=\PP(\xi^{j}=\xi\mid\xi^{i}=\xi)\cdot\PP(\xi^{i}=\xi).
\end{align}
With the facts that $\PP(\xi^{j}=\xi\mid\xi^{i}=\xi)=p_{\xi}+O(\lambda^{j-i})$
and
$\PP(\xi^{i}=\xi)=p_{\xi}+O(\lambda^i)$,
\begin{align}
    \EE\big(\textbf{1}_{\xi}(\xi^i)\textbf{1}_{\xi}(\xi^j)\big)=p_{\xi}^2+O(\max\{\lambda^{i},\lambda^{j-i}\}).
\end{align}
Obviously, it holds
\begin{align}\label{uniformbetter-t1}
    \sum_{i<j\leq k}\max\{\lambda^{i},\lambda^{j-i}\}=\sum_{t=1}^{\lceil \frac{k}{2}\rceil} c_t\lambda^t.
\end{align}
Now, we investigate what $c_t$ exact is. For any $1\leq t\leq \lceil\frac{k}{2}\rceil$, $\lambda^t$ only appears in the cases (I) $i=t$ and $j-i\geq t$ or (II) $j-i=t$ and $i\geq t$. Thus, we can get
\begin{align}
c_{t}\leq \sharp(\textrm{I})+\sharp(\textrm{II})=k-2t+k-2t=2k-4t.
\end{align}
Thus, we derive
\begin{align}\label{uniformbetter-t2}
    \sum_{i<j\leq k}\max\{\lambda^{i},\lambda^{j-i}\}\leq\sum_{t=1}^{\lceil \frac{k}{2}\rceil} (2k-4t)\lambda^t=O(\frac{k}{1-\lambda}).
\end{align}
With \eqref{uniformbetter-t1} and \eqref{uniformbetter-t2}, we then get
\begin{align}
    c)=(k^2-k)p_{\xi}+O(\frac{k}{1-\lambda}).
\end{align}
Substituting the bounds of a), b) and c) to \eqref{uniformbetter-t0}, we get
\begin{align}
    \EE\|c_{\xi}(k)-k\cdot p_{\xi}\|^2=O(\frac{k}{1-\lambda}).
\end{align}
Thus, the expectation of noise is bounded as
\begin{align}
    \EE\|e^k\|^2=O(\frac{A}{k(1-\lambda)}).
\end{align}
By Slight modification of proof of inexact MC-BCD, we then prove the result.
\end{proof}

\begin{figure} 
\begin{center}
    \includegraphics[width=0.4\textwidth]{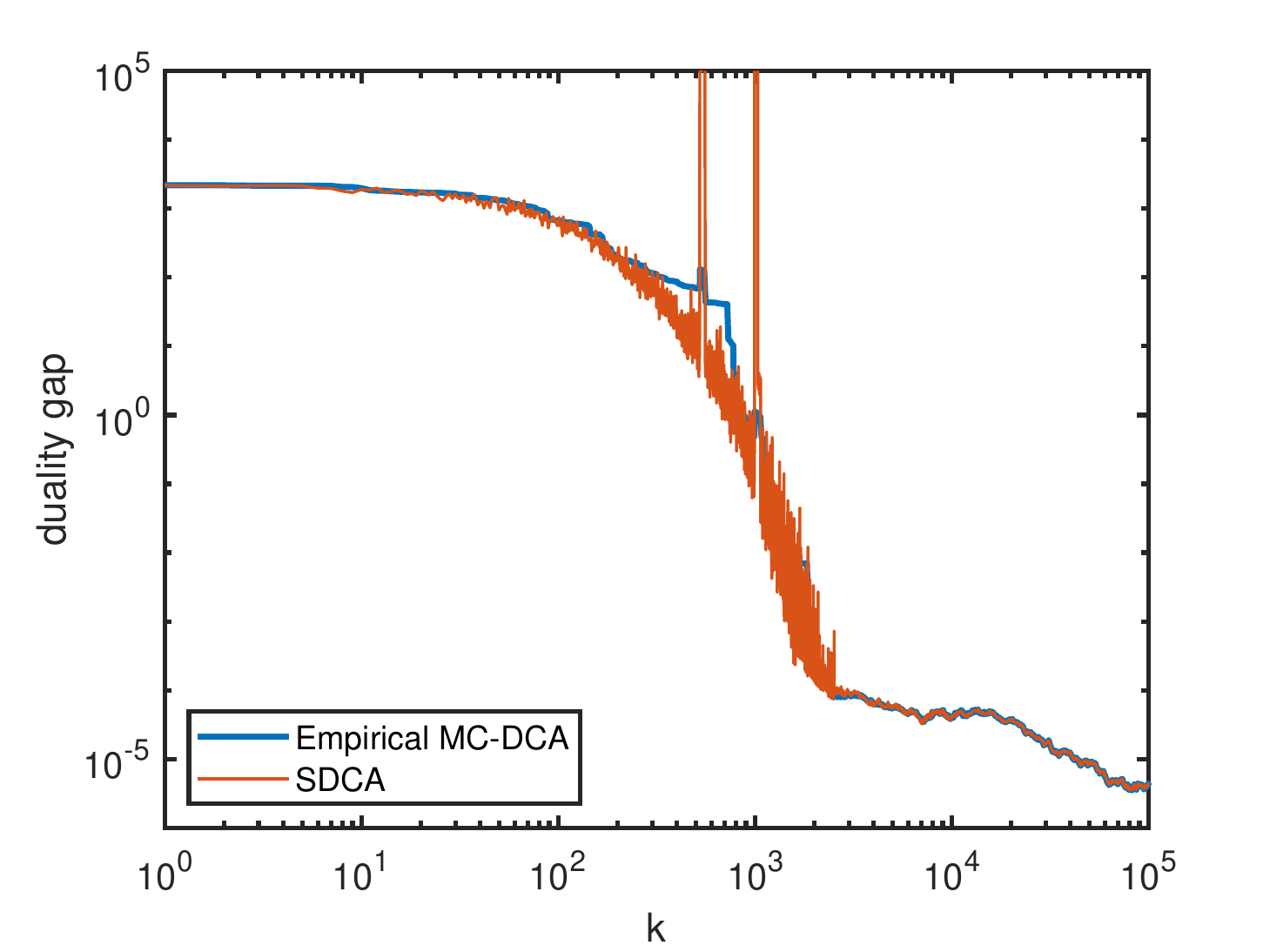}
\end{center}
\caption{Duality gap after $k$ samples and $k$ iterations of the algorithms. Empirical MC-DCA runs each iteration along with sampling. SDCA obtains all samples first and then runs $k$ iterations with $(\bar{p}_{\xi})_{\xi\in\Xi}$ estimated from the $k$ samples.
}\label{fig:empirical_MCDCA}
\end{figure}

We also use a numerical experiment to verify the convergence of empirical MC-DCA and comparison with SDCA. We created a 40-state Markov chain with non-uniform stationary distribution. We randomly generated $x\in\mathbb{R}^20$, $\xi_i\in\mathbb{R}^{20}, i=1,\ldots,40$, and set $b_i=\xi_i^{\top}x$, where $i$ is a state of the Markov chain. We also set $F_i(x)=x-b_i$ and $\lambda=0.1$.  We compare duality gap of empirical MC-DCA and SDCA when doing same number of samples and iterations. The MC-DCA runs each iteration along with sampling, which SDCA does sampling first and then do minimization with $(\bar{p}_{\xi})_{\xi\in\Xi}$ estimated from the samples. Figure \ref{fig:empirical_MCDCA} shows that empirical MC-DCA can reach the same convergence rate as SDCA. However, empirical MC-DCA can minimize along sampling and does not need to store the sample space in memory. It can reach any accuracy as long as the sampling process continues. However, SDCA requires the knowledge of the sample space at each iteration. To improve the accuracy, it must resume the sampling process to re-estimate $(p_{\xi})_{\xi\in\Xi}$.

\section{Conclusion}
In summary, we propose a new class of BCD method which can be implemented by visiting a random sequence of nodes in a network. As long as the network is connected, the method can run without the knowledge of its topology and other global parameters. Besides networks, our method can be also used for certain Markov decision processes. It can also run along with MCMC samples for empirical risk minimization when the underlying distribution cannot be sampled directly. The convergence of our method is proved for both convex and nonconvex objective functions with constant stepsize. Inexact subproblems are allowed. When the objective is convex and strongly convex, sublinear and linear convergence rates are proved, respectively.
\section*{Acknowledgments}
The work of Y. Sun and W. Yin is supported in part by NSF DMS-1720237 and ONR N000141712162. The work of Y. Xu is supported in part by NSF DMS-1719549.  

\end{document}